\documentclass[a4paper,10pt]{amsart}
\usepackage[utf8]{inputenc}
\usepackage{amsxtra}
\usepackage{amsopn}
\usepackage{amsmath,amsthm,amssymb}
\usepackage{amscd}
\usepackage{color}
\usepackage{amsfonts}
\usepackage{latexsym}
\usepackage{verbatim}
\usepackage{color}

\theoremstyle{plain}
\newtheorem{theorem}{Theorem}[section]
\newtheorem*{theorem*}{Theorem}

\newtheorem{lemma}[theorem]{Lemma}
\newtheorem{prop}[theorem]{Proposition}

\newtheorem{rem}[theorem]{Remark}
\newtheorem{ex}[theorem]{Example}
\newtheorem*{mt*}{Main Theorem}


\newcommand\R{{\mathbb R}}

\newcommand{\del}{\partial}
\newcommand{\delbar}{\overline{\del}}

\title[$\delbar$-Harmonic forms on $4$-dimensional almost-Hermitian manifolds]{$\delbar$-Harmonic forms on $4$-dimensional almost-Hermitian manifolds}
\author[Nicoletta Tardini and Adriano Tomassini]{Nicoletta Tardini and Adriano Tomassini}
\address{Dipartimento di Scienze Matematiche, Fisiche e Informatiche\\
Unit\`{a} di Matematica e Informatica,
Universit\`{a} degli Studi di Parma\\
Parco Area delle Scienze 53/A, 43124 \\
Parma, Italy}
\email{nicoletta.tardini@gmail.com}
\email{nicoletta.tardini@unipr.it}
\email{adriano.tomassini@unipr.it}

\thanks{\newline 
The first author is partially supported by GNSAGA of INdAM and has financially been supported by the Programme ``FIL-Quota Incentivante'' of University of Parma and co-sponsored by Fondazione Cariparma.
The second author is partially supported by the Project PRIN 2017 ``Real and Complex Manifolds: Topology, Geometry and holomorphic dynamics'' and by GNSAGA of INdAM}
\begin{document}

\maketitle

\begin{abstract}
Let $(X,J)$ be a $4$-dimensional compact almost-complex manifold and let $g$ be a Hermitian metric on $(X,J)$. Denote by $\Delta_{\delbar}:=\delbar\delbar^*+\delbar^*\delbar$ the $\delbar$-Laplacian. If $g$ is {\em globally conformally K\"ahler}, respectively {\em (strictly) locally conformally K\"ahler}, we prove that the dimension of the space of $\delbar$-harmonic $(1,1)$-forms on $X$, denoted as $h^{1,1}_{\delbar}$, is a topological invariant given by $b_-+1$, respectively $b_-$. As an application, we provide a one-parameter family of almost-Hermitian structures on the Kodaira-Thurston manifold for which such a dimension is $b_-$. This gives a positive answer to a question raised by T. Holt and W. Zhang. Furthermore, the previous example shows that $h^{1,1}_{\delbar}$ depends on the metric, answering to a Kodaira and Spencer's problem.
Notice that such almost-complex manifolds admit both almost-K\"ahler and (strictly) locally conformally K\"ahler metrics and this fact cannot occur on compact complex manifolds.
\end{abstract}

\section{Introduction}

Let $(X, J)$ be an almost-complex manifold, then if $J$ is not integrable one has that $\delbar^2\neq 0$ and so the Dolbeault cohomology of $X$
$$
H^{\bullet,\bullet}_{\delbar}(X):=
\frac{\text{Ker}\,\delbar}{\text{Im}\,\delbar}
$$
is not well defined.\\
However, if $g$ is a Hermitian metric on $(X,J)$ and $*$ denotes the associated Hodge-$*$-operator, then
$$
\Delta_{\delbar}:=\delbar\delbar^*+\delbar^*\delbar
$$
is a well-defined second order, elliptic, differential operator, without assuming the integrability of $J$. In particular, if $X$ is compact, then $\text{Ker}\Delta_{\delbar}$
is a finite-dimensional vector space and we will denote as usual with $h^{\bullet,\bullet}_{\delbar}$ its dimension.
 If $J$ is integrable, then the $(p,q)$-Dolbeault cohomology groups $
H^{p,q}_{\delbar}(X)$ of the compact complex manifold $(X,J)$ are isomorphic to the Kernel of $\Delta_{\delbar}$, that is 
$$
H^{p,q}_{\delbar}(X)\simeq \text{Ker}\Delta_{\delbar}\vert_{A^{p,q}(X)}\,,
$$
where $A^{p,q}(X)$ denotes the space of smooth $(p,q)$-forms on $(X,J)$
and, in particular, the dimension of the space of $\delbar$-harmonic forms is a holomorphic invariant, not depending on the choice of the Hermitian metric.
In \cite[Problem 20]{hirzebruch} Kodaira and Spencer asked whether this is the case also when $J$ is not integrable. More precisely,
\smallskip

\noindent
{\bf Question I} {\em Let $(M,J)$ be an almost-complex manifold. Choose a Hermitian metric on $(M,J)$ and consider the numbers $h^{p,q}_{\delbar}$. Is $h^{p,q}_{\delbar}$ independent of the choice of the Hermitian metric?}

In \cite{holt-zhang} Holt and Zhang answered negatively to this question, showing that there exist almost-complex structures on the Kodaira-Thurston manifold such that the Hodge number $h^{0,1}_{\delbar}$ varies with different choices of Hermitian metrics. \newline
Furthermore, in \cite[Proposition 6.1]{holt-zhang} the authors showed that for a compact $4$-dimensional almost-K\"ahler manifold $h^{1,1}_{\delbar}$ is independent of the metric, and more precisely $h^{1,1}_{\overline\partial}=b_-+1$, where $b_-$ denotes the dimension of the space of the anti-self-dual harmonic $2$-forms. In \cite[Question 6.2]{holt-zhang} the authors asked the following \smallskip

\noindent
{\bf Question II} {\em Let $(M,J)$ be a compact almost-complex $4$-dimensional manifold which admit an almost-K\"ahler structure. Does it have a non almost-K\"ahler Hermitian metric such that 
$$
h^{1,1}_{\delbar}\neq b_-+1 ?
$$
}

In this paper we study this problem. In fact, first we show that on compact almost-complex $4$-dimensional manifolds $h^{1,1}_{\delbar}$ is a conformal invariant of Hermitian metrics (see Lemma \ref{lemma:conformal}).
In particular, this means that \cite[Proposition 6.1]{holt-zhang} can be extended to Hermitian metrics that are globally conformal to an almost-K\"ahler metric (for simplicity we will call these metrics globally conformally K\"ahler, even though the almost-complex structure may not be integrable).\\
Next we show, using the existence (and uniqueness up to omotheties) of the Gauduchon representative in every conformal class, the following, see Theorem \ref{thm:dim4-gauduchon},
\vskip.2truecm
\noindent
{\bf Theorem 1}
{\em 
Let $(X^4,J)$ be a compact almost-complex manifold of dimension $4$, then,
with respect to a (strictly) locally conformally K\"ahler metric,
$$
h^{1,1}_{\delbar}=b_-.
$$
}
Here, by ({\em strictly}) {\em locally conformally K\"ahler metric} we mean a Hermitian metric $\omega$, such that
$$
d\omega=\theta\wedge\omega
$$
with $\theta$ a $d$-closed, non $d$-exact, differential $1$-form.\\
Then we show the following, see Theorem \ref{thm:dim4-lck-gck},
\vskip.2truecm
\noindent
{\bf Theorem 2}\ {\em 
Let $(X^4,J)$ be a compact almost-complex manifold of dimension $4$ and let $\omega$ be a Hermitian metric, then 
if $\omega$ is globally conformally K\"ahler (in particular if it is almost-K\"ahler), it holds
$$
h^{1,1}_{\delbar}=b_-+1.
$$
If $\omega$ is (strictly) locally conformally K\"ahler, 
$$
h^{1,1}_{\delbar}=b_-.
$$
}
In particular, for locally conformally K\"ahler and globally conformally K\"ahler metrics on compact $4$-dimensional almost-complex manifolds, $h^{1,1}_{\delbar}$ is a topological invariant. Notice that this was already known in the integrable case, by \cite[Proposition II.6]{gauduchon}.\\
We also show in Proposition \ref{prop:lower-bound} that on every compact $4$-dimensional almost-Hermitian manifold, $b_-$ is a lower bound for $h^{1,1}_{\delbar}$, that is optimal in view of Theorem \ref{thm:dim4-gauduchon}.\\
Finally, we discuss these results on explicit examples
on the Kodaira-Thurston manifold, denoted with $X$. This is a compact $2$-step nilmanifold of dimension $4$ that can be endowed with both complex and symplectic structures but cannot admit any K\"ahler metrics (see \cite{kodaira}, \cite{thurston}) and it has the structure of a principal $\mathbb{S}^1$-bundle over a $3$-torus. This manifold turns out to be a very valuable source of examples in non-K\"ahler geometry.\\
First, in example \ref{ex:kt-lck}, we construct a family of almost-complex structures $J_a$, with $a\in\mathbb{R}\setminus\left\lbrace 0\right\rbrace$, $a^2<1,$ on $X$ that admit both almost-K\"ahler and (strictly) locally conformally K\"ahler metrics 
(notice that in view of \cite{vaisman-no-lck} this could not happen in the integrable case). Namely, $(X,J_a)$ is a compact almost-complex $4$-dimensional manifold which admit an almost-K\"ahler metric $\tilde\omega_a$ and a non almost-K\"ahler Hermitian metric $\omega_a$ such that
$$
h^{1,1}_{\delbar}=b_-\neq b_-+1\,.
$$
Hence, this example answers affirmatively to \cite[Question 6.2]{holt-zhang} in the case of the Kodaira-Thurston manifold endowed with the $1$-parameter family of almost-complex structures $J_a$. \\
Moreover, this answers to Kodaira and Spencer's question, showing that also the Hodge number $h^{1,1}_{\delbar}$ depends on the Hermitian metric and not just on the almost-complex structure.\\
Then, in example \ref{ex:kt-gck} we construct on $X$ a left-invariant
almost-complex structure compatible with a family of non left-invariant globally conformally K\"ahler metrics. In particular, in this case we will have
$$
h^{1,1}_{\delbar}=b_-+1\,.
$$
In both examples we write down explicitly the $\delbar$-harmonic representatives in $\mathcal{H}^{1,1}_{\delbar}$.

\medskip

\section{Preliminaries}
\label{preliminaries}
In this Section we recall some basic facts about almost-complex and almost-Hermitian manifolds and fix some notations.
Let $X$ be a smooth manifold of dimension $2n$ and let $J$ be an almost-complex structure on $X$, namely a $(1,1)$-tensor on $X$ such that $J^2=-\text{Id}$. Then, $J$ induces on the space of forms $A^\bullet(X)$ a natural bigrading, namely
$$
A^\bullet(X)=\bigoplus_{p+q=\bullet}A^{p,q}(X)\,.
$$
Accordingly, the exterior derivative $d$ splits into four operators
$$
d:A^{p,q}(X)\to A^{p+2,q-1}(X)\oplus A^{p+1,q}(X)\oplus A^{p,q+1}(X)\oplus A^{p-1,q+2}(X)
$$
$$
d=\mu+\del+\delbar+\bar\mu\,,
$$
where $\mu$ and $\bar\mu$ are differential operators that are linear over functions. In particular, they are related to the Nijenhuis tensor $N_J$ by
$$
\left(\mu\alpha+\bar\mu\alpha\right)(u,v)=\frac{1}{4} \alpha\left(N_J(u,v)\right)
$$
where $\alpha\in A^1(X)$. Hence, $J$ is integrable, that is $J$ induces a complex structure on $X$, if and only if $\mu=\bar\mu=0$.\\
In general, since $d^2=0$ one has
$$
\left\lbrace
\begin{array}{lcl}
\mu^2 & =& 0\\
\mu\del+\del\mu & = & 0\\
\del^2+\mu\delbar+\delbar\mu & = & 0\\
\del\delbar+\delbar\del+\mu\bar\mu+\bar\mu\mu & = & 0\\
\delbar^2+\bar\mu\del+\del\bar\mu & = & 0\\
\bar\mu\delbar+\delbar\bar\mu & = & 0\\
\bar\mu^2 & =& 0
\end{array}
\right.\,.
$$
In particular, $\delbar^2\neq 0$ and so the Dolbeault cohomology of $X$
$$
H^{\bullet,\bullet}_{\delbar}(X):=
\frac{\text{Ker}\,\delbar}{\text{Im}\,\delbar}
$$
is well defined if and only if $J$ is integrable.\\
If $g$ is a Hermitian metric on $(X,J)$ with fundamental form $\omega$ and $*$ is the associated Hodge-$*$-operator, one can consider the following differential operator
$$
\Delta_{\delbar}:=\delbar\delbar^*+\delbar^*\delbar\,.
$$
This is a second order, elliptic, differential operator and we will denote its kernel by
$$
\mathcal{H}^{p,q}_{\delbar}(X):=\text{Ker}\,\Delta_{\delbar_{\vert A^{p,q}(X)}}\,.
$$
In particular, if $X$ is compact this space is finite-dimensional and its dimension will be denoted by $h^{p,q}_{\delbar}$. By \cite{holt-zhang} we know that these numbers are not holomorphic invariants, more precisely they depend on the choice of the Hermitian metric. When needed we will use the notations $\mathcal{H}^{p,q}_{\delbar,\omega}$, $h^{p,q}_{\delbar,\omega}$ in order to stress on the dependence on the Hermitian metric $\omega$.\\
However, if the Hermitian metric is almost-K\"ahler and $2n=4$, in  \cite[Proposition 6.1]{holt-zhang} it was shown that $h^{1,1}_{\delbar}=b_-+1$, depends only on the topology of $X$.\\
We recall that one can consider also other elliptic differential operators on almost-Hermitian manifolds, as done in \cite{tardini-tomassini}, as generalizations of the classical Dolbeault, Bott-Chern and Aeppli Laplacians defined on complex manifolds.\\
Moreover, a generalization of the Dolbeault cohomology in the non integrable setting was introduced and studied in \cite{cirici-wilson-1} and \cite{cirici-wilson-2}.
\medskip
\noindent

Let us fix now an almost-Hermitian metric $g$ on a compact $2n$-dimensional almost-complex manifold $(X,J)$, which we will identify with its associated $(1,1)$-form $\omega$. 
Then $J$  acts as an isomorphism on $\wedge^{p,q}X$ by $J \alpha = \sqrt{-1}^{q-p}\alpha$, $\alpha\in\wedge^{p,q}X$. Via this extension, it follows that $J^2=(-1)^k\mathrm{id}$, so that $J^{-1}=(-1)^kJ=J^*$ on $\wedge^kX$, where $J^*$ is the pointwise adjoint of $J$ with respect to some (and so any) Hermitian metric. We denote by $d^c$ the differential operator $d^c:=-J^{-1}dJ$.\\
Then, we can consider
the linear operator $L:=\omega\wedge\_$ and its adjoint $\Lambda:=L^*$. We recall that $L^{n-1}\colon \wedge^1X \to \wedge^{2n-1}X$ is an isomorphism, therefore one can define the {\em Lee form} of $\omega$, as:
$$
\theta:= \Lambda d\omega = Jd^{*}\omega \in\wedge^1X
$$
such that
$$ d\omega^{n-1}=\theta\wedge\omega^{n-1}. $$
We will say that $\omega$ is \emph{(strictly) locally conformally K\"ahler} if
$$
d\omega=\alpha\wedge\omega
$$
where $\alpha$ is a $d$-closed, non $d$-exact, $1$-form. 
In particular, in this case, the Lee form of $\omega$ is
$$
\theta=\frac{1}{n-1}\alpha.
$$
The metric $\omega$ will be called \emph{globally conformally K\"ahler} if 
$$
d\omega=\alpha\wedge\omega
$$
with $\alpha$ $d$-exact $1$-form. Indeed, if $\alpha=df$ then the metric $e^{-f}\omega$ is almost-K\"ahler.\\
Notice that, if $\tilde\omega=\Phi\omega$, with $\Phi\in\mathcal{C}^\infty(X,\mathbb{R})$, $\Phi>0$, are two conformal Hermitian metrics, then the associated Lee forms are related by
$$
\theta_{\tilde\omega}=\theta_{\omega}+(n-1)d\,\text{log}\,\Phi\,,
$$
in particular, $d\theta_{\tilde\omega}=d\theta_{\omega}$. Hence, all the Hermitian metrics conformal to a (strictly) locally conformally K\"ahler are still (strictly) locally conformally K\"ahler.\\
Another important class of Hermitian metrics is given by the \textit{Gauduchon metrics}, which are defined by  $dd^c\omega^{n-1}=0$ or equivalently as having co-closed Lee form. These metrics are a very useful tool in conformal and almost-Hermitian geometry, in view of the celebrated result by Gauduchon, \cite[Th\'eor\`eme 1]{gauduchon-CRAS}, which states that if $(M,J)$ is an $n$-dimensional compact almost-complex manifold with $n>1$, then any conformal class of any given almost-Hermitian metric contains a Gauduchon metric, unique up to multiplication with positive constants.

\section{$h^{1,1}_{\delbar}$ on compact almost-Hermitian $4$-dimensional manifolds}\label{main}

In this section we study the Hodge number $h^{1,1}_{\delbar}$ on compact almost-Hermitian $4$-dimensional manifolds.\\
We first show in arbitrary dimension the following Lemma, that ensures that in suitable degrees the Hodge numbers are conformal invariants.

\begin{lemma}\label{lemma:conformal-dim-n}
Let $(X^{2n},J)$ be a compact almost-complex manifold of dimension $2n$, then, for $p+q=n$, $h^{p,q}_{\delbar}$ is a conformal invariant of Hermitian metrics.
\end{lemma}
\begin{proof}
Let $\tilde\omega=\Phi\omega$ be two conformal Hermitian metrics, with $\Phi$ smooth positive function on $X$. Then, for $(p,q)$-forms on a $2n$-dimensional manifold we have that the associated Hodge-$*$-operators are related by,
$$
*_{\tilde\omega}=\Phi^{n-p-q}*_{\omega}.
$$
In general we would have that $\psi\in A^{1,1}(X)$ is $\delbar$-harmonic with respect to $\tilde\omega$ if and only if
$$
\delbar\psi=0,\qquad \del*_{\tilde\omega}\psi=0
$$
if and only if
$$
\delbar\psi=0,\qquad \del(\Phi^{n-p-q}*_{\omega}\psi)=0
$$
Now we compute
\begin{eqnarray*}
\del(\Phi^{n-p-q}*_{\omega}\psi)&=&
\del\Phi^{n-p-q}\wedge *_{\omega}\psi+\Phi^{n-p-q}\del*_{\omega}\psi\\
&=&(n-p-q)\Phi^{n-p-q-1}\del\Phi\wedge *_{\omega}\psi+\Phi^{n-p-q}\del*_{\omega}\psi.
\end{eqnarray*}
Clearly, if $p+q=n$ we have that
$$
*_{\tilde\omega}=*_{\omega}.
$$
and
$$
\del*_{\tilde\omega}\psi=0\quad\iff\quad
\del*_{\omega}\psi=0.
$$
Hence, for $p+q=n$
$$
\mathcal{H}^{p,q}_{\delbar,\Phi\omega}=\mathcal{H}^{p,q}_{\delbar,\omega}.
$$
In particular, their dimensions coincide,
$$
h^{p,q}_{\delbar,\Phi\omega}=h^{p,q}_{\delbar,\omega}.
$$
\end{proof}

As a corollary we have the following

\begin{lemma}\label{lemma:conformal}
Let $(X^4,J)$ be a compact almost-complex manifold of dimension $4$, then $h^{1,1}_{\delbar}$ is a conformal invariant of Hermitian metrics.
\end{lemma}

As an immediate application we have the following
\begin{prop}
Let $(X^4,J)$ be a compact almost-complex manifold of dimension $4$, then,
with respect to a globally conformally K\"ahler metric,
$$
h^{1,1}_{\delbar}=b_-+1.
$$
\end{prop}
\begin{proof}
Since $h^{1,1}_{\delbar}$ is a conformal invariant, the result follows by \cite{holt-zhang}. Indeed, for almost-K\"ahler metrics $h^{1,1}_{\delbar}=b_-+1$.
\end{proof}

We first prove the following (cf. \cite[Proposition II.6]{gauduchon} for the integrable case)
\begin{prop}\label{prop:f-const}
Let $(X^4,J)$ be a compact almost-complex manifold of dimension $4$ and let $\omega$ be a Gauduchon metric, then the trace of a $\delbar$-harmonic $(1,1)$-form is constant. Namely, if $\psi\in\mathcal{H}^{1,1}_{\delbar}$ is written as
$$
\psi=f\omega+\gamma\qquad\text{with\,}*\gamma=-\gamma\,,
$$
then $f$ is constant.
\end{prop}
\begin{proof}
Let $\psi\in A^{1,1}(X)$ be a $\delbar$-harmonic $(1,1)$-form. Then,
$$
\psi=f\omega+\gamma
$$
with $\gamma$ anti-self dual $(1,1)$-form, namely $*\gamma=-\gamma$, and $f=\frac{1}{2}\text{tr\,}\psi=\frac{1}{2}\Lambda\psi=\frac{1}{2}\langle\psi,\omega\rangle$.
Hence,
$$
*\psi=f\omega-\gamma.
$$
Since $\psi$ is $\delbar$-harmonic , then $\delbar\psi=0$ and $\del*\psi=0$, namely
$$
\delbar(f\omega)=-\delbar\gamma
$$
and
$$
\del(f\omega)=\del\gamma.
$$
Recalling that for $(1,1)$-forms on a $4$-dimensional manifolds we have that 
$$d^c=i(\delbar-\del)
$$ 
and 
$$
dd^c+d^cd=0;
$$
summing up the previous two equations we get
$$
d(f\omega)=id^c\gamma,
$$
hence
$$
dd^c(f\omega)=0.
$$
Since $\omega$ is Gauduchon this implies that $f$ is constant. For completeness we recall here the proof (cf. for instance also the proof in \cite[Theorem 10]{angella-istrati-otiman-tardini}). Since, on a $4$-dimensional manifold $d\omega=\theta\wedge\omega$, where $\theta$ is the Lee form of $\omega$, we have
\begin{eqnarray*}
dd^c(f\omega)&=&d(d^cf\wedge\omega+fJ\theta \wedge\omega)\\
&=&dd^cf\wedge\omega-d^cf\wedge\theta\wedge\omega +df\wedge J\theta \wedge\omega \\
&&+fdJ\theta \wedge\omega +f\theta \wedge J\theta \wedge\omega \\
&=&\left(dd^cf-d^cf\wedge\theta +df\wedge J\theta +fdJ\theta +f\theta \wedge J\theta \right)\wedge\omega\\
&=&\Lambda \left(dd^cf-d^cf\wedge\theta +df\wedge J\theta +fdJ\theta +f\theta \wedge J\theta \right)\frac{\omega^2}{2}.
\end{eqnarray*}

Therefore, $dd^c(f\omega )=0$ is equivalent to:
$$
\Lambda \left(dd^cf-d^cf\wedge\theta +df\wedge J\theta +fdJ\theta +f\theta \wedge J\theta \right)=0
$$
Recall now that from \cite[Lemma 7]{angella-istrati-otiman-tardini} on an almost-Hermitian manifold we have, for every $1$-form $\alpha$,
$$
\Lambda(dJ\alpha)=-d^*\alpha-\langle\alpha,\theta\rangle.
$$
Therefore,
$$
\Lambda \left(dd^cf-d^cf\wedge\theta +df\wedge J\theta +fdJ\theta +f\theta \wedge J\theta \right)=0
$$
if and only if
$$
\Lambda dd^cf-\Lambda(d^cf\wedge\theta)+\Lambda(df\wedge J\theta)-
fd^*\theta-f\vert\theta\vert^2+f\Lambda(\theta\wedge J\theta)=0.
$$
Since, $\omega$ is Gauduchon we have $d^*\theta=0$, hence the last equation is equivalent to
$$
-\Delta f-\langle df,\theta\rangle
-\Lambda(d^cf\wedge\theta)+\Lambda(df\wedge J\theta)-
fd^*\theta-f\vert\theta\vert^2+f\Lambda(\theta\wedge J\theta)=0.
$$
This holds if and only if
$$
-\Delta f-\langle df,\theta\rangle+2\langle df,\theta\rangle-
f\vert\theta\vert^2+f\vert\theta\vert^2=0.
$$
Therefore, we have obtained that $dd^c(f\omega )=0$ if and only if
$$
-\Delta f+\langle df,\theta\rangle=0.
$$
Namely, $f\in\text{Ker}\,L^*$ where, for Gauduchon metrics,
$$
L^*(h)=\Delta h-\langle dh,\theta\rangle
$$
is the adjoint of the operator $L$, with
$$
L(h)=\Delta h+\langle dh,\theta\rangle.
$$
Now, by \cite{gauduchon} (cf. also \cite{gauduchon-CRAS}) $f$ is either positive or negative (unless $f=0$). Suppose that $f>0$ (otherwise one can argue with $-f$), then by $dd^c(f\omega)=0$ we have that $f\omega$ is a Gauduchon metric conformal to $\omega$, and so $f$ is constant.
\end{proof}

\begin{rem}\label{rem:theta-d*-exact}
{\rm
In the proof of the previous proposition we had that
$$
d(f\omega)=id^c\gamma.
$$
Since, $f$ is constant and $d\omega=\theta\wedge\omega$ we have that applying the Hodge-$*$-operator,
$$
f*(\theta\wedge\omega)=i*d^c\gamma.
$$
Since $\theta$ is a primitive form, one has that $*(\theta\wedge\omega)=*L\theta=-J\theta$. Moreover, using that $J\gamma=\gamma$ and $*\gamma=-\gamma$ one obtain that
$$
f\theta=-id^*\gamma.
$$
In particular, if $\theta\neq 0$ (i.e., $\omega$ is not almost-K\"ahler), we have that
$$
f=0\qquad\iff\qquad d^*\gamma=0
\qquad\iff\qquad\gamma\text{ is harmonic.}
$$
}\end{rem}

As a consequence, we prove that with respect to (strictly) locally conformally K\"ahler structures, $h^{1,1}_{\delbar}$ is a topological invariant.

\begin{theorem}\label{thm:dim4-gauduchon}
Let $(X^4,J)$ be a compact almost-complex manifold of dimension $4$
and suppose that there exists a
(strictly) locally conformally K\"ahler metric, then
$$
h^{1,1}_{\delbar}=b_-.
$$
\end{theorem}

\begin{proof}
Let $\tilde\omega$ be a (strictly) locally conformally K\"ahler metric on $(X^4,J)$.
Since, by Lemma \ref{lemma:conformal}, $h^{1,1}_{\delbar}$ is a conformal invariant we fix in the conformal class of $\tilde\omega$ the Gauduchon representative $\omega$ of volume $1$. Clearly, $\omega$ is still (strictly) locally conformally-K\"ahler.\\
Let $\psi\in A^{1,1}(X)$ be a $\delbar$-harmonic $(1,1)$-form. Then,
$$
\psi=f\omega+\gamma
$$
with $*\gamma=-\gamma$ and $f$ constant by Proposition \ref{prop:f-const}.
By Remark \ref{rem:theta-d*-exact},
$$
f\theta=d^*(-i\gamma)\in\text{Im\,}d^*.
$$
Now, we want to show that $f=0$. Suppose by contradiction that $f\neq 0$, hence
$$
\theta=d^*\left(-\frac{i}{f}\gamma\right)\in\text{Im\,}d^*,
$$
but, since $\omega$ is still (strictly) locally conformally-K\"ahler, $d\theta=0$.
So $\theta$ is $d$-closed and $d^*$-exact and so $\theta=0$, but this is absurd since $d\omega\neq 0$.
Therefore, $f=0$ and
$$
\psi=\gamma\qquad\text{with}\qquad *\gamma=-\gamma
$$
with $d^*\gamma=0$, that is $\gamma$ harmonic, concluding the proof.
\end{proof}

%

In particular, as a consequence of 
Lemma \ref{lemma:conformal}, Theorem \ref{thm:dim4-gauduchon} and \cite[Proposition 6.1]{holt-zhang}, for locally conformally K\"ahler and globally conformally K\"ahler metrics on compact $4$-dimensional almost-complex manifolds, $h^{1,1}_{\delbar}$ is a topological invariant. Namely, we have proven the following

\begin{theorem}\label{thm:dim4-lck-gck}
Let $(X^4,J)$ be a compact almost-complex manifold of dimension $4$ and let $\omega$ be a Hermitian metric, then 
if $\omega$ is globally conformally K\"ahler (in particular if it is almost-K\"ahler), it holds
$$
h^{1,1}_{\delbar}=b_-+1.
$$
If $\omega$ is (strictly) locally conformally K\"ahler, 
$$
h^{1,1}_{\delbar}=b_-.
$$
\end{theorem}

Notice that in the integrable case, $h^{1,1}_{\delbar}$ only depends on the complex structure and for compact complex surfaces this result is known (cf. \cite[Proposition II.6]{gauduchon}). Indeed, recall that a compact complex surface is K\"ahler if and only $b_1$ is even and in this case $h^{1,1}_{\delbar}=b_-+1$. On the other side, a compact complex surface is non-K\"ahler if and only $b_1$ is odd and in this case $h^{1,1}_{\delbar}=b_-$.
This is coherent with our result since on compact K\"ahler surfaces there exist no (strictly) locally conformally K\"ahler metrics by \cite{vaisman-no-lck}. In fact, this last statement holds more generally, indeed by (\cite[Theorem 2.1, Remark (1)]{vaisman-no-lck}, on compact complex manifolds satisfying the $\del\delbar$-lemma every locally conformally K\"ahler structure is also globally conformally K\"ahler. 
We want to point out that the first non-integrable examples of almost-K\"ahler manifolds admitting (strictly) locally conformally K\"ahler manifolds appeared in \cite{vaisman-lcak}. In Section \ref{section:examples} we will construct a new family of examples on the Kodaira-Thurston manifold.\\
In view of these results, we ask whether there exist examples of Hermitian metrics on compact almost-complex $4$-dimensional manifolds with $h^{1,1}_{\delbar}$ different from $b_-$ and $b_-+1$.\\

The following result gives a general estimate for $h^{1,1}_{\delbar}$.
\begin{prop}\label{prop:lower-bound}
Let $(X^4,J)$ be a compact almost-complex manifold of dimension $4$ and let $\omega$ be a Hermitian metric, then
$$
\mathcal{H}_g^-\subseteq \mathcal{H}^{1,1}_{\delbar},
$$
where $\mathcal{H}_g^-$ denotes the space of anti-self-dual harmonic $(1,1)$-forms.\\
In particular,
$$
h^{1,1}_{\delbar}\geq b_{-}\,.
$$
\end{prop}
\begin{proof}
Let $\gamma$ be an anti-self-dual $(1,1)$-form, namely $*\gamma=-\gamma$. Suppose that $\gamma$ is harmonic, that is equivalent to $d\gamma=0$. Hence, for degree reasons
$$
\delbar\gamma=0\qquad\text{and}\qquad \del*\gamma=-\del\gamma=0.
$$
Therefore, $\gamma$ is $\delbar$-harmonic.
\end{proof}

In particular, Theorem \ref{thm:dim4-lck-gck} shows that the minimum $h^{1,1}_{\delbar}$ is reached by strictly locally conformally K\"ahler metrics.\\

\section{Explicit constructions on the Kodaira-Thurston manifold}\label{section:examples}

In this section we apply the results obtained in Section \ref{main}  to construct explicit examples.
First, we recall the definition of the Kodaira-Thurston manifold $X$. Let
$$
\mathbb{H}_3(\mathbb{R}):=
\left\lbrace
\left[\begin{matrix}
1 &  x_1 & x_3\\
0 & 1 & x_2\\
0 & 0 &1\\
\end{matrix}\right]
\mid x_1,x_2,x_3\in\mathbb{R}\right\rbrace\,
$$
be the $3$-dimensional Heisenberg group and let $\Gamma$ be the subgroup of $\mathbb{H}_3(\mathbb{R})$ of the matrices with integral entries. Then,
$$
X:=\frac{\mathbb{H}_3(\mathbb{R})}{\Gamma}\times \mathbb{S}^1
$$
is a compact $4$-dimensional manifold admitting both complex and symplectic structures but no K\"ahler structures. Denoting with $x_4$ the coordinate on $\mathbb{S}^1$, a global frame on $X$ is given by
$$
e_1:=\partial_{x_1}\,,\quad
e_2:=\partial_{x_2}+x_1\partial_{x_3}\,,\quad
e_3:=\partial_{x_3}\,,\quad
e_4:=\partial_{x_4}\,,
$$
and its dual coframe is
$$
e^1:=dx_1\,,\quad
e^2:=dx_2\,,\quad
e^3:=dx_3-x_1dx_2\,,\quad
e^4:=dx_4.
$$
In particular, the only non trivial bracket is $\left[e_1,e_2\right]=e_3$
and so the structure equations become
$$
de^1=de^2=de^4=0\,,\quad
de^3=-e^1\wedge e^2\,.
$$
In the sequel we denote by $e^{ij}=e^i\wedge e^j$ and similarly. \\

Then
$$
H^2_{dR}(X;\R)\simeq\hbox{\rm Span}_\R\langle [e^{13}-e^{24}],[e^{14}+e^{23}],[e^{13}+e^{24}],[e^{14}-e^{23}]\rangle
$$
where all the representatives are harmonic, with respect to the Riemannian metric $g=\displaystyle\sum_{j=1}^4 e^j\otimes e^j$. Furthemore, the space of $d$-harmonic anti-self dual forms is isomorphic to  
$$
\hbox{\rm Span}_\R\langle [e^{13}+e^{24}],[e^{14}-e^{23}]\rangle,
$$
so that $b_-(X)=2$.

\begin{ex}\label{ex:kt-lck}{\rm
Now we construct the following family of almost-complex structures $J_a$ on $X$, with $a\in\mathbb{R}$, setting as coframe of $(1,0)$-forms
$$
\Phi^1_a:=(e^1+ae^4)+ie^3\,,\qquad
\Phi^2_a:=e^2+ie^4\,.
$$
We will use the notation $\Phi^1=\Phi^1_a$, $\Phi^2=\Phi^2_a$.
The dual $(1,0)$-frame of vector fields is given by
$$
V_1:=\frac{1}{2}(e_1-ie_3)\,,\qquad
V_2:=\frac{1}{2}(e_2-i(e_4-ae_1)).
$$
One can show directly that the complex structure equations become
$$
\begin{aligned}
d\Phi^1=&-\frac{i}{4}\Phi^{12}-\frac{i}{4}\Phi^{1\bar 2}+
\frac{i}{4}\Phi^{2\bar1}+\frac{a}{2}\Phi^{2\bar2}-\frac{i}{4}\Phi^{\bar1\bar2}\,,\\
d\Phi^2=& 0,
\end{aligned}
$$
and in particular, $J_a$ is a non integrable almost-complex structure.
For every $a\in\mathbb{R}$,
we fix the following Hermitian metric
$$
\omega_a:=\frac{i}{2}\left(\Phi^1\wedge\bar\Phi^1+
\Phi^2\wedge\bar\Phi^2\right).
$$
A direct computation gives that
$$
d\omega_a=i\frac{a}{4}\Phi^{12\bar2}-i\frac{a}{4}\Phi^{2\bar1\bar2}=\theta_a\wedge\omega_a
$$
with $\theta_a=\frac{a}{2}(\Phi^1+\bar\Phi^1)$. In particular,
\begin{enumerate}
 \item[$\bullet$]$d\theta_a=0$,
 \item[$\bullet$] $\omega_a$ is an almost-K\"ahler metric if and only if $a=0$.
\end{enumerate}
%
Therefore, for $a\neq 0$, the Lee form $\theta_a$ of the almost-Hermitian metric $\omega_a$ is closed and not $d$-exact, hence $\omega_a$ is a strictly locally conformally K\"ahler metric.\\
Notice that, for $a^2<1$, $J_a$ admits a compatible almost-K\"ahler metric, indeed
$$
\tilde\omega_a:=\omega_a+\frac{a}{2}\Phi^{1\bar 2}-\frac{a}{2}\Phi^{2\bar 1}
$$
is a Hermitian metric such that
$$
d\tilde\omega_a=0.
$$
Hence, by Theorem \ref{thm:dim4-lck-gck} for the Hodge numbers we have on $(X,J_a)$, with $a^2<1$,
\begin{itemize}
 \item $
h^{1,1}_{\delbar,\omega_a}=b_-=2\qquad \hbox{\rm for}\,a\neq 0
$\\
\item $
h^{1,1}_{\delbar,\omega_0}=b_-+1=3
$\\
\item $
h^{1,1}_{\delbar,\tilde\omega_a}=b_-+1=3
$
\end{itemize}

This, in particular, partially answers to Question 6.2 in \cite{holt-zhang}, giving explicit examples of compact almost-complex $4$-dimensional manifolds which admit an almost-K\"ahler metric and also admit a non almost-K\"ahler Hermitian metric with $h^{1,1}_{\delbar}\neq b_-+1$.\\
Moreover, this answers to Kodaira and Spencer's question, showing that also the Hodge number $h^{1,1}_{\delbar}$ depend on the Hermitian metric and not just on the almost-complex structure.\\
For the sake of completeness we write down the PDE's system that one should solve in order to find a basis for $\mathcal{H}^{1,1}_{\delbar,\omega_a}$.\\
Let $\psi\in A^{1,1}(X)$ be an arbitrary $(1,1)$-form on $X$, then $\psi$ can be written as
$$
\psi=A \Phi^{1\bar1}+B \Phi^{1\bar2}+L \Phi^{2\bar1}+M \Phi^{2\bar2}
$$
where $A,\,B,\,L,\,M$ are smooth functions on $X$.\\

By the complex structure equations, we get
$$
\delbar\psi=\bar V_2(A) \Phi^{1\bar1\bar2}
-\bar V_1(B) \Phi^{1\bar1\bar2}+
\bar V_2(L) \Phi^{2\bar1\bar2}-
\bar V_1(M) \Phi^{2\bar1\bar2}-
\frac{a}{2}A \Phi^{2\bar1\bar2}+
\frac{i}{4}B \Phi^{2\bar1\bar2}-
\frac{i}{4}L \Phi^{2\bar1\bar2}\,,
$$
hence $\delbar\psi=0$ if and only if
$$
\begin{aligned}
\bar V_2(A)-\bar V_1(B)=&\,0\,,\\
\bar V_2(L)-\bar V_1(M)-\frac{a}{2}A+\frac{i}{4}B-\frac{i}{4}L=&\,0\,.
\end{aligned}
$$
Now, if we denote with $*$ the Hodge-$*$-operator with respect to the metric $\omega_a$, we have
$$
*\psi=M \Phi^{1\bar1}-B \Phi^{1\bar2}-
L \Phi^{2\bar1}+A \Phi^{2\bar2}
$$
and
$$
\del*\psi=-V_2(M) \Phi^{12\bar1}+
V_2(B) \Phi^{12\bar2}-
V_1(L) \Phi^{12\bar1}+V_1(A) \Phi^{12\bar2}+
\frac{a}{2}M\Phi^{12\bar2}+
\frac{i}{4}B \Phi^{12\bar2}-
\frac{i}{4}L \Phi^{12\bar2}\,,
$$
hence $\del*\psi=0$ if and only if
$$
\begin{aligned}
V_2(M)+ V_1(L)=&\,0\,,\\
V_2(B)+V_1(A)+\frac{a}{2}M+\frac{i}{4}B-\frac{i}{4}L=&\,0\,.
\end{aligned}
$$
Therefore, $\psi$ is harmonic if and only if 
$$
\left\{
\begin{aligned}
\bar V_2(A)-\bar V_1(B)=&\,0\,,\\
\bar V_2(L)-\bar V_1(M)-\frac{a}{2}A+\frac{i}{4}B-\frac{i}{4}L=&\,0\,,\\
V_2(M)+ V_1(L)=&\,0\,,\\
V_2(B)+V_1(A)+\frac{a}{2}M+\frac{i}{4}B-\frac{i}{4}L=&\,0\,.
\end{aligned}
\right.
$$
Since we know that $h^{1,1}_{\delbar,\omega_a}=b_-=2$, the solution of this system is given by $A,B,L,M$ constants and
$$
A=-M=\frac{i}{2a}B-\frac{i}{2a}L.
$$
Hence,
$$
\mathcal{H}^{1,1}_{\delbar,\omega_a}=\mathbb{C}\left\langle
\frac{i}{2a}\Phi^{1\bar1}+\Phi^{1\bar2}-\frac{i}{2a}\Phi^{2\bar2}\,,
-\frac{i}{2a}\Phi^{1\bar1}+\Phi^{2\bar1}+\frac{i}{2a}\Phi^{2\bar2}
\right\rangle\,.
$$
}\end{ex}

\medskip

Since in the integrable case, on K\"ahler manifolds every locally conformally K\"ahler metric is globally conformally K\"ahler we want to put in evidence the following
\begin{prop}
The Kodaira-Thurston manifold with the almost-complex structure $J_a$ constructed above admits both almost-K\"ahler metrics and (strictly) locally conformally K\"ahler metrics.
\end{prop}
For other examples we refer to \cite{vaisman-lcak}.

\begin{ex}\label{ex:kt-gck}
{\rm
Now we define a different, non left-invariant, Hermitian structure on the Kodaira-Thurston manifold $X$.\\
First we define a non-integrable left-invariant complex structure $J$ on $X$ setting
as global co-frame of $(1,0)$-forms
$$
\varphi^1:=e^1+ie^3\,,\qquad
\varphi^2:=e^2+ie^4
$$
and the corresponding structure equations become
$$
d\varphi^1=-\frac{i}{4}\varphi^{12}-\frac{i}{4}\varphi^{1\bar 2}+
\frac{i}{4}\varphi^{2\bar 1}-\frac{i}{4}\varphi^{\bar 1\bar2}\,,\quad
d\varphi^2=0.
$$
We will denote with $\left\lbrace W_1,W_2\right\rbrace$ its dual frame, more precisely
$$
W_1:=\frac{1}{2}\left(e_1-ie_3\right)\,,\quad
W_2:=\frac{1}{2}\left(e_2-ie_4\right)\,,
$$
namely,
$$
W_1=\frac{1}{2}\left(\del_{x_1}-i\del_{x_3}\right)\,,\quad
W_2=\frac{1}{2}\left(\del_{x_2}+x_1\del_{x_3}-i\del_{x_4}\right).
$$
We consider now on $(X,J)$ the following $1$-parameter family of almost-Hermitian metrics
$$
\omega_{tf}=e^{2tf(x_2)}e^1\wedge e^3+e^2\wedge e^4
$$
where $f=f(x_2)$ is a $\mathbb{Z}$-periodic smooth function with $e_2(f)\neq 0$.
In particular, $f$ induces a smooth function on $X$.
Then it is immediate to check that the almost-complex structure $J$ is compatible with $
\omega_{tf}$ so that it defines a positive definite Hermitian metric on $X$.
In fact, the metric $\omega_{tf}$ is almost-K\"ahler if and only if $t=0$. Indeed, by the structure equations and the fact that $e_2(f)=f'(x_2)\neq 0$ by assumption, we have
$$
d\omega_{tf}=-2te^{2tf(x_2)}e_2(f)e^{123}.
$$
Hence $\omega_{tf}$ is a Hermitian deformation of an almost-K\"ahler metric on $X$.\\
Moreover, notice that 
$$
d\omega_{tf}=\theta_{tf}\wedge\omega_{tf},
$$
with $\theta_{tf}=2tf'(x_2)e^2=2t\,df$, namely the Lee form of $\omega_{tf}$ is $d$-exact, which means that $\omega_{tf}$ is globally conformally K\"ahler.
A direct computation shows that, in fact
$$
e^{-2tf(x_2)}\omega_{tf}
$$
is an almost-K\"ahler metric. Hence, by Theorem \ref{thm:dim4-lck-gck}
for the Hodge numbers we have on $(X,J)$,
$$
h^{1,1}_{\delbar,\omega_{tf}}=b_-+1=3.
$$
For completeness we write down the system that one should solve in order to find a basis for $\mathcal{H}^{1,1}_{\delbar,\omega_{tf}}$.\\
Let $\psi\in A^{1,1}(X)$ be an arbitrary $(1,1)$-form on $X$, then $\psi$ can be written as
$$
\psi=A \varphi^{1\bar1}+B \varphi^{1\bar2}+L \varphi^{2\bar1}+M \varphi^{2\bar2}
$$
where $A,\,B,\,L,\,M$ are smooth functions on $X$.\\
By the complex structure equations, we get
$$
\delbar\psi=\bar W_2(A)\varphi^{1\bar1\bar2}
-\bar W_1(B)\varphi^{1\bar1\bar2}+
\bar W_2(L)\varphi^{2\bar1\bar2}-
\bar W_1(M)\varphi^{2\bar1\bar2}+
\frac{i}{4}B\varphi^{2\bar1\bar2}-
\frac{i}{4}L\varphi^{2\bar1\bar2}\,,
$$
hence $\delbar\psi=0$ if and only if
$$
\begin{aligned}
\bar W_2(A)-\bar W_1(B)=&\,0\,,\\
\bar W_2(L)-\bar W_1(M)+\frac{i}{4}B-\frac{i}{4}L=&\,0\,.
\end{aligned}
$$
Now, if we denote with $*$ the Hodge-$*$-operator with respect to the metric $\omega_{tf}$ 
we have that a unitary frame is given by
$$
\psi_1=\frac{1}{\sqrt 2}e^{tf(x_2)}\varphi^1\,,\quad
\psi^2=\frac{1}{\sqrt 2}\varphi^2
$$
and hence we have
$$
*\psi=Ae^{-2tf(x_2)}\varphi^{2\bar2}-B\varphi^{1\bar2}-
L\varphi^{2\bar1}+Me^{2tf(x_2)}\varphi^{1\bar1}
$$
and
$$
\del*\psi=-W_2(Me^{2tf(x_2)})\varphi^{12\bar1}+
W_2(B)\varphi^{12\bar2}-
W_1(L)\varphi^{12\bar1}+W_1(Ae^{-2tf(x_2)})\varphi^{12\bar2}+
\frac{i}{4}B\varphi^{12\bar2}-
\frac{i}{4}L\varphi^{12\bar2}\,,
$$
hence $\del*\psi=0$ if and only if
$$
\begin{aligned}
W_2(Me^{2tf})+ W_1(L)=&\,0\,,\\
W_2(B)+W_1(Ae^{-2tf(x_2)})+\frac{i}{4}B-\frac{i}{4}L=&\,0\,.
\end{aligned}
$$
Therefore, $\psi$ is harmonic if and only if
$$
\left\{
\begin{aligned}
\bar W_2(A)-\bar W_1(B)=&\,0\,,\\
\bar W_2(L)-\bar W_1(M)+\frac{i}{4}B-\frac{i}{4}L=&\,0\,,\\
W_2(Me^{2tf(x_2)})+ W_1(L)=&\,0\,,\\
W_2(B)+W_1(Ae^{-2tf(x_2)})+\frac{i}{4}B-\frac{i}{4}L=&\,0\,.
\end{aligned}
\right.
$$
Since we know that 
$$h^{1,1}_{\delbar,\omega_{tf}}=b_-+1=3,$$ 
the solution of this system is given by $A$ complex constant, $B=L$ complex constants and $$M=M_0\,e^{-2tf(x_2)}$$ where $M_0$ is a complex constant.
Hence,
$$
\mathcal{H}^{1,1}_{\delbar,\omega_{tf}}=\mathbb{C}\left\langle
\varphi^{1\bar1}\,,
\varphi^{1\bar2}+\varphi^{2\bar1}\,,
e^{-2tf(x_2)}\varphi^{2\bar2}
\right\rangle\,.
$$
}\end{ex}
\begin{rem}{\rm
It has to be remarked that solving these kind of PDE's systems is not an easy task. Indeed, as a general method, one could use Fourier analysis to expand the unknown complex valued functions $A,B,L,M$, obtaining a first order ODE's system on the Fourier coefficients of $A,B,L,M$, which, as far as we know, is very challenging to solve (cf. also \cite{holt-zhang,holt-zhang-2} for further comments).
}\end{rem}

\end{document}